\documentclass[a4paper,11pt, reqno]{amsart}
\usepackage{amssymb,amsmath}
\usepackage{cite}
\usepackage{amsmath,amscd,amssymb}
\usepackage{latexsym}
\usepackage[colorlinks,citecolor=blue,pagebackref,hypertexnames=false]{hyperref}

\numberwithin{equation}{section}
\allowdisplaybreaks[2]
\theoremstyle{plain}
\newtheorem{theorem}{Theorem}[section]

\newtheorem{lemma}[theorem]{Lemma}

\newtheorem{proposition}[theorem]{Proposition}

\theoremstyle{definition}
\newtheorem{definition}[theorem]{Definition}

\theoremstyle{remark}
\newtheorem{remark}[theorem]{Remark}

\newtheorem{case[theorem]}{Case}

\def \C{{\mathbb C}}

\def\supp{\hbox{supp\,}}
\def\norm#1.#2.{\lVert#1\rVert_{#2}}

\title[Decay estimates for a class of semigroups on metric measure spaces]{Decay estimates for a class of semigroups related to self-adjoint operators on metric measure spaces
}

\author{Guoxia Feng}
\author{Manli Song}
\thanks{Manli Song is the corresponding author.}
\author{Huoxiong Wu}

\address{\endgraf School of Mathematical Sciences, Xiamen University, Xiamen 361005, China}
\address{\endgraf School of Mathematics and Statistics, Northwestern Polytechnical University, Xi'an, Shaanxi 710129, China}
\email{gxfeng@mail.nwpu.edu.cn}

\address{\endgraf Research\&Development Institute of Northwestern Polytechnical University in Shenzhen, Sanhang Science\&Technology Building, No. 45th, Gaoxin South 9th Road, Nanshan District, Shenzhen City, 518063}
\address{\endgraf School of Mathematics and Statistics, Northwestern Polytechnical University, Xi'an, Shaanxi 710129, China}
\email{mlsong@nwpu.edu.cn}
\address{\endgraf School of Mathematical Sciences, Xiamen University, Xiamen 361005, China}
\email{huoxwu@xmu.edu.cn}
\keywords{Decay estimates; non-negative self-adjoint operators; metric measure space; Hermite operator; twisted Laplacian; Laguerre operator; Strichartz estimates.}
\subjclass[2010]{Primary 42B37, 47D08.}

\date{\today}
\begin{document}
	
	\maketitle

	\allowdisplaybreaks

	\begin{abstract}
		Assume that $(X,d,\mu)$ is a metric space endowed with a non-negative Borel measure $\mu$ satisfying the doubling condition and the additional condition that $\mu(B(x,r))\gtrsim r^n$ for any $x\in X, \,r>0$ and some $n\geq1$. Let $L$ be a non-negative self-adjoint operator on $L^2(X,\mu)$. We assume that $e^{-tL}$ satisfies a Gaussian upper bound and the Schr\"odinger operator $e^{itL}$ satisfies an $L^1\to L^\infty$ decay estimate of the form
		\begin{equation*}
			\|e^{itL}\|_{L^1\to L^\infty} \lesssim |t|^{-\frac{n}{2}}.
		\end{equation*}
		Then for a general class of dispersive semigroup $e^{it\phi(L)}$, where $\phi: \mathbb{R}^+ \to \mathbb{R}$ is smooth, we establish a similar $L^1\to L^\infty$ decay estimate by a suitable subordination formula connecting it with the Schr\"odinger operator $e^{itL}$. As applications, we derive new Strichartz estimates for several dispersive equations related to Hermite operators, twisted Laplacians and Laguerre operators.
	\end{abstract}
	\tableofcontents 	
	
	\section{Introduction}
	Let $(X,d,\mu)$ be a metric space endowed with a non-negative Borel measure $\mu$ satisfying the doubling condition: there exists a constant $C>0$ such that
	\begin{equation}\label{doubling-condition}
		\mu(B(x,2r))\leq C\mu(B(x,r)),
	\end{equation}
	for all $x\in X$, $r>0$ and all balls $B(x,r):=\{y\in X:d(x,y)<r\}$. In addition, we shall assume in this paper that
	\begin{equation}\label{addition}
		\mu(B(x,r))\gtrsim r^n,
	\end{equation}
	for all $x\in X$, $r>0$ and for some $n\geq1$.
	
	From the doubling property \eqref{doubling-condition}, it yields a constant $D>0$ such that
	\begin{equation*}
		\mu(B(x,\lambda r))\leq C\lambda^D\mu(B(x,r)),
	\end{equation*}
	for all $\lambda\geq1$, $x\in X$ and $r>0$; and that
	\begin{equation*}
		\mu(B(x,r))\leq C\left(1+\frac{d(x,y)}{r}\right)^{\tilde{n}}\mu(B(y,r)),
	\end{equation*}
	for all $x,y\in X$ and $r>0$.

	Suppose $L$ is a non-negative self-adjoint operator on $L^2(X,\mu)$ and we shall make the following assumptions on $L$:
	
	(A1) The Schr\"odinger operator satisfies an $L^1\to L^\infty$ decay estimate:
	\begin{equation*}
		\|e^{itL}\|_{L^1\to L^\infty} \lesssim |t|^{-\frac{n}{2}},\,|t|<T_0,
	\end{equation*}
	where $T_0\in(0,+\infty]$.
	
	(A2) The kernel $p_t(x,y)$ of $e^{-tL}$ admits a Gaussian upper bound: there exist $C,c>0$ such that for all $x,y\in X$ and $t>0$,
	\begin{equation*}
		|p_t(x,y)|\leq \frac{C}{\mu(B(x,\sqrt{t}))}\exp\left(-\frac{d(x,y)^2}{ct}\right).
	\end{equation*}
	This is frequently the case for many important operators, notably the Laplacian $L=-\Delta$ on the Euclidean space and its potential perturbations, such as the Hermite operator, the twisted Laplacian and the Laguerre operator, ect.
	
	Under (A1)-(A2) assumptions, the key ingredient of this paper is to establish a decay estimate for a general class of dispersive semigroup $e^{it\phi(L)}$, where $\phi: \mathbb{R}^+ \to \mathbb{R}$ is smooth satisfying:
	\\
	
	(H1)~There exists $m_1>0$, such that for any $\alpha \geqslant 2$ and $\alpha \in \mathbb{N}$,
	\begin{equation*}
		|\phi'(r)| \sim r^{m_1-1}  , \quad |\phi^{(\alpha)}(r)| \lesssim r^{m_1-\alpha},\quad r \geqslant 1.
	\end{equation*}
	
	(H2)~There exists $m_2>0$, such that for any $\alpha \geqslant 2$ and $\alpha \in \mathbb{N}$,
	\begin{equation*}
		|\phi'(r)| \sim r^{m_2-1}  , \quad |\phi^{(\alpha)}(r)| \lesssim r^{m_2-\alpha},\quad 0<r<1.
	\end{equation*}
	
	(H3)~There exists $\alpha_1$, such that
	\begin{equation*}
		|\phi''(r)| \sim r^{\alpha_1-2}, \quad r \geqslant1.
	\end{equation*}
	
	(H4)~There exists $\alpha_2>0$, such that
	\begin{equation*}
		|\phi''(r)| \sim r^{\alpha_2-2}, \quad  0<r<1.
	\end{equation*}
	\begin{remark}
		The assumptions (H1)-(H4) root in \cite{GPW2008}. (H1) and (H3) represent the homogeneous order of $\phi$ in high frequency and guarantee $\alpha_1\leq m_1$. Similarly, the homogeneous order of $\phi$ in low frequency is described by (H2) and (H4) which also make sure  $\alpha_2\geq m_2$.
	\end{remark}
	
	In the past decades, Strichartz estimates have been very useful in the study of nonlinear partial differential equations. These estimates in the Euclidean setting have been proved for many dispersive equations, such as the wave equation and Schr\"{o}dinger equation (see \cite{GV, KT, Str}). To obtain Strichartz estimates, it involves basically two types of ingredients. The first one consists in estimating the $L^1\to L^\infty$ decay in time on the evolution group associated with the free equation. The second one consists of abstract arguments, which are mainly duality arguments. Therefore, the $L^1\to L^\infty$ decay estimate plays a crucial role.
	
	In 2008, Guo-Peng-Wang\cite{GPW2008} used a unified way to study the decay for a class of dispersive semigroup $e^{it\phi(\sqrt{-\Delta})}$ on the Euclidean space. Many dispersive wave equations reduce to this type, for instance, the Schr\"{o}dinger equation ($\phi(r)=r^2$), the wave equation ($\phi(r)=r$), the fractional Schr\"{o}dinger equation ($\phi(r)=r^\nu$) ($0<\nu<2$), the Klein-Gordon equation ($\phi(r)=\sqrt{1+r^2}$), the beam equation ($\phi(r)=\sqrt{1+r^4}$) and the fourth-order Schr\"{o}dinger equation ($\phi(r)=r^2+r^4$), etc. When $\phi$ is a homogeneous function of order $m$, namely, $\phi(\lambda r)=\lambda^m\phi(r),\,\forall \lambda>0$, the dispersive estimate can be easily obtained by a theorem of Littman and dyadic decomposition. However, it becomes very complicated when $\phi$ is not homogeneous since the scaling constants can not be effectively separated from the time.  To overcome the difficulty, they applied frequency localization by separating $\phi$ between high and low frequency in different scales. Guo-Peng-Wang\cite{GPW2008} assumed that $\phi: \mathbb{R}^+ \to \mathbb{R}$ is smooth satisfying (H1)-(H4) and obtained the $L^1\to L^\infty$ decay estimates for the semigroup $e^{it\phi(\sqrt{-\Delta})}$ which requires elaborate techniques involving oscillatory integrals, the vanishing properties of Bessel functions at the origin and the infinity and the stationary phase theorem.
	
	Many authors are also interested in adapting the well known Strichartz estimates from the Euclidean setting to a more abstract setting, such as the Sch\"odinger operator and wave operator on H-type groups and more generality of step $2$ stratified Lie groups, see \cite{BGX2000,H2005,FMV1, FV,BKG,LS2014,SZ,Song2016,BBG2021}, etc. Inspired by Guo-Peng-Wang \cite{GPW2008}, Song-Yang \cite{SY2023} discussed the decay for a general class of dispersive semigroup $e^{it\phi(\mathcal{L})}$ on the Heisenberg group, where $\mathcal{L}$ is the sub-Laplacian on the Heisenberg group. 
	
	Strichartz inequalities have already been studied in the literature on metric measure spaces, which includes some notable examples like the Schr\"odinger operator associated with the Hermite operator on $\mathbb{R}^n$ \cite{NR2005}, the special Hermite operator on $\mathbb{C}^n$ \cite{R2008} and the Laguerre operator on $\mathbb{R}_+^{n}$ \cite{S2013}, etc. When $L$ is the Hermite operator or the twisted Laplacian and $\phi(L)=\sqrt{L}$, D\text{'}Ancona-Pierfelice-Ricci \cite{DPR2010} deduced a decay estimate for the wave operator $e^{it\sqrt{L}}$ directly from a corresponding estimate for the Schr\"odinger operator $e^{itL}$ by a suitable subordination formula connecting these two operators. Recently, under (A1)-(A2) assumption, Bui-D\text{'}Ancona-Duong-M\"{u}ller\cite{BDXM2019} extended D\text{'}Ancona-Pierfelice-Ricci's results to more general semigroups $e^{it\phi(L)}$ on metric measure spaces, where $\phi$ is with power-like behavior near $0$ and $\infty$.
	\begin{theorem}[see \cite{BDXM2019}]
		Assume $L$ satisfies (A1) and (A2), and $\psi\in C^\infty(\mathbb{R})$ is supported in $[1/2,2]$. 
		\begin{itemize}
			\item[(1)]  If $\phi$ satisfies (H1) and (H3), and in addition $0<m_1=\alpha_1\leq1$, then
			\begin{equation*} \left|\psi(\lambda^{-1}\sqrt{L})e^{it\phi(L)}f\right|\lesssim{|t|}^{-\frac{n-1}{2}}\lambda^{(1-m_1)n+m_1}\left\|f\right\|_{L^1},\,\lambda\geq1, |t|<T_0.
			\end{equation*}
			\item[(2)]  If $\phi$ satisfies (H2) and (H4), and in addition $m_2=\alpha_2$, then
			\begin{equation*} \left|\psi(\lambda^{-1}\sqrt{L})e^{it\phi(L)}f\right|\lesssim{|t|}^{-\frac{n-1}{2}}\lambda^{(1-m_2)n+m_2}\left\|f\right\|_{L^1},\,0<\lambda<1, |t|<T_0.
			\end{equation*}
		\end{itemize}
	\end{theorem}
	
	The particularly interesting case in \cite{BDXM2019} is the fractional powers $\phi(L)=L^\nu,\,\nu\in(0,1)$. We shall consider a more general semigroups $e^{it\phi(L)}$ which can also deal with the nonhomogeneous case like $\phi(L)=\sqrt{1+L}$, $\phi(L)=\sqrt{1+L^2}$ and $\phi(L)=L+L^2$ corresponding the Klein-Gordon equation, the beam equation and the fourth-order Schr\"odinger equation, respectively. Our main results are the following.
	\begin{theorem}\label{main-result}
		Assume $L$ satisfies (A1) and (A2), and $\psi\in C^\infty(\mathbb{R})$ is supported in $[1/2,2]$. 
		\begin{itemize}
			\item[(1)]  If $\phi$ satisfies (H1), then
			\begin{equation*} \left|\psi(\lambda^{-1}\sqrt{L})e^{it\phi(L)}f\right|\lesssim{|t|}^{-\frac{n-2}{2}}\lambda^{(1-m_1)n+2m_1}\left\|f\right\|_{L^1},\,\lambda\geq1, |t|<T_0, n\geq2;
			\end{equation*}
			In addition, if $\phi$ satisfies (H3), then
			\begin{equation*} \left|\psi(\lambda^{-1}\sqrt{L})e^{it\phi(L)}f\right|\lesssim{|t|}^{-\frac{n-1}{2}}\lambda^{(1-m_1)n+2m_1-\alpha_1}\left\|f\right\|_{L^1},\,\lambda\geq1, |t|<T_0, n\geq1.
			\end{equation*}
			\item[(2)]  If $\phi$ satisfies (H2), then
			\begin{equation*} \left|\psi(\lambda^{-1}\sqrt{L})e^{it\phi(L)}f\right|\lesssim{|t|}^{-\frac{n-2}{2}}\lambda^{(1-m_2)n+2m_2}\left\|f\right\|_{L^1},\,0<\lambda<1, |t|<T_0, n\geq2;
			\end{equation*}
			In addition, if $\phi$ satisfies (H4), then
			\begin{equation*} \left|\psi(\lambda^{-1}\sqrt{L})e^{it\phi(L)}f\right|\lesssim{|t|}^{-\frac{n-1}{2}}\lambda^{(1-m_2)n+2m_2-\alpha_2}\left\|f\right\|_{L^1},\,0<\lambda<1, |t|<T_0, n\geq1.
			\end{equation*}
		\end{itemize}
	\end{theorem}
	
	The plan of the paper is the following. In Section 2, we recall some preliminary kernel estimates, van der Corput's lemma and the Besov spaces associated to the non-negative self-adjoint operator $L$. In Section 3, we prove the crucial subordination formulas, extending the earlier results by M\"uller-Seeger \cite{Muller-Seeger2015} or Bui-D\text{'}Ancona-Duong-M\"{u}ller\cite{BDXM2019}. In Section 4, we apply the previous subordination formulas to obtain general decay estimates for the semigroup $e^{it\phi(L)}$. In the final Section, we examine some applications to concrete operators, namely the Hermite operator, the twisted Laplacian and the Laguerre operator.
	\section{Preliminary results}
	\subsection{Kernel estimates}
	Let $L$ be a non-negative self-adjoint operator on $L^2(X)$ satisfying (A2). Denote by $E_L(\lambda)$ the spectral decomposition of $L$. Then by spectral theory (see \cite{Hulanicki}), for any bounded Borel function $F: [0,\infty)\to\mathbb{C}$, we can define
	\begin{equation*}
		F(L)=\int_0^\infty F(\lambda)dE_L(\lambda).
	\end{equation*}
	We have the following useful lemma.
	\begin{lemma}[see \cite{BDXM2019}]
		$(a)$ Let $\varphi\in\mathcal{S}(\mathbb{R})$ be an even function. Then for any $N>0$ there exists $C$ such that the kernel of $\varphi(t\sqrt{L})$ satisfies
		\begin{equation*}
			\left|K_{\varphi(t\sqrt{L})}(x,y)\right|\leq\frac{C}{\mu(B(x,t))+\mu(B(y,t))}\left(1+\frac{d(x,y)}{t}\right)^{-N},
		\end{equation*}
		for all $t>0$ and $x,y\in X$.
		
		$(b)$ Let $\varphi\in\mathcal{S}(\mathbb{R})$ be an even function. Then for any $1\leq p\leq  q\leq\infty$ we have
		\begin{equation}\label{L^p-L^q}
			\|\varphi(t\sqrt{L})\|_{L^p\to L^q}\lesssim t^{-\left(\frac{n}{p}-\frac{n}{q}\right)},\,t>0.
		\end{equation}
	\end{lemma}
	\subsection{Van der Corput's Lemma}
	We recall van der Corput's Lemma.
	\begin{lemma}(see \cite{S1993})\label{van-der-Corput} Let $g\in C^\infty([a,b])$ be real-valued such that
		\begin{equation*}
			|g''(x)|\geq \delta
		\end{equation*}
		for any $x\in[a,b]$ with $\delta >0$. Then for any function $\psi \in C^\infty([a,b])$, there exists a constant $C$ which does not depend on $\delta, a, b, g$ or $\psi$, such that
		\begin{equation*}
			\left|\int_a^b e^{itg(x)}\psi(x)\,dx\right|\leq Ct^{-1/2}\left(\|\psi\|_\infty+\|\psi'\|_1\right).
		\end{equation*}
	\end{lemma}
	\subsection{Besov space related to the non-negative self-adjoint operator $L$}  Let $\Psi\in C^\infty(\mathbb{R})$ be an even function such that $0\leq \Psi\leq1, \Psi=1$ in $[0,1]$ and $\Psi=0$ in $[2,\infty)$. Let $\psi(r)=\Psi(r)-\Psi(2r)$ such that ${\rm{supp}}\, \psi\subseteq\{r\in\mathbb{R}: 1/2<|r|<2\}$ and generate a Littlewood-Paley dyadic partition of unity $\{\psi_j\}_{j\in\mathbb{Z}}$ on $\mathbb{R}$, where $\psi_j(r)=\psi(2^{-j}r),\,\forall j\in\mathbb{Z}$, with
	\begin{align*}
		\sum_{j\in\mathbb{Z}}\psi_j(r)&=1,\,\forall r>0,\\
		\Psi(r)+\sum_{j=1}^\infty\psi_j(r)&=1,\,\forall r\geq0.
	\end{align*}
	\begin{definition}\label{homo-def}
		For all $s\in\mathbb{R}$, $1\leq p,q\leq\infty$, define the homogeneous Besov spaces $\dot{B}_{p,q}^{s,L}$ related to the non-negative self-adjoint operator $L$ as the completion of the set
		\begin{equation*}
			\left\{f\in L^2(X,\mu): \|f\|_{\dot{B}_{p,q}^{s,L}}<\infty\right\}
		\end{equation*}
		for the norm $\|\cdot\|_{\dot{B}_{p,q}^{s,L}}$ given by
		\begin{equation*}
			\|f\|_{\dot{B}_{p,q}^{s,L}}:=\left(\sum_{j\in\mathbb{Z}}2^{jsq}\|\psi_j(\sqrt{L})\|^q_{L^p}\right)^\frac{1}{q}.
		\end{equation*}
	\end{definition}
	\begin{definition}\label{nonhom-def}
		For all $s\in\mathbb{R}$, $1\leq p,q\leq\infty$, define the non-homogeneous Besov spaces $B_{p,q}^{s,L}$ related to the non-negative self-adjoint operator $L$ as the completion of the set
		\begin{equation*}
			\left\{f\in L^2(X,\mu): \|f\|_{B_{p,q}^{s,L}}<\infty\right\}
		\end{equation*}
		for the norm $\|\cdot\|_{B_{p,q}^{s,L}}$ given by
		\begin{equation*}
			\|f\|_{B_{p,q}^{s,L}}:=\|\Psi(\sqrt{L})f\|_{L^p}+\left(\sum_{j=1}^\infty 2^{jsq}\|\psi_j(\sqrt{L})\|^q_{L^p}\right)^\frac{1}{q}.
		\end{equation*}
	\end{definition}
	\begin{remark}\label{Besov-property}
		(1) We note that the homogeneous Besov space in Definition \ref{homo-def} and the non-homogeneous Besov space in Definition \ref{nonhom-def} are independent of the choice of $\Psi$. See \cite{BDXM2019}.
		
		(2) When $L$ is the Hermite operator, the twisted Laplacian or the Laguerre operator whose spectrum is strictly positive and discrete, we can drop from the sum in the Definition \ref{homo-def} the terms with dyadic $2^j$ for $j\leq 0$ (assuming in addition that $\Psi(r)=1$ in $[0,3/2]$). In this case, the homogeneous Besov space $\dot{B}_{p,q}^{s,L}$ coincides with the non-homogeneous Besov space $B_{p,q}^{s,L}$. 
	\end{remark}
	\section{Decay estimates for $e^{it\phi(L)}$}
	\subsection{Subordination formulas}
	In this section, we obtain the crucial subordination formulas connecting the functions $e^{it\phi(x)}$ and $e^{itx}$, which are strongly inspired by ideas from M\"uller-Seeger \cite{Muller-Seeger2015} or Bui-D\text{'}Ancona-Duong-M\"{u}ller\cite{BDXM2019}.
	\begin{theorem}\label{High-frequecy}
		Assume $\phi$ satisfies (H1) and $g\in C^\infty(\mathbb{R})$ is supported in $[1/2,2]$. Then there exist $c_0>1$, and functions $\rho_t(x,\lambda)$ and $a_t(s,\lambda)$ defined on $\mathbb{R}^2$ for each $t$ satisfying
		\begin{equation} \label{rho}
			{\rm{supp}}\ \rho_t(\cdot,\lambda)\subset[\lambda^2/5,5\lambda^2]\quad \text{and} \quad |\rho_t(x,\lambda)|\leq C(k,\|g\|_{C^k},\phi)(t\lambda^{2m_1})^{-k},\, k\geq0,
		\end{equation}
		and
		\begin{equation}\label{a}
			{\rm{supp}}\ a_t(\cdot,\lambda)\subset[2c_0^{-1},2c_0]\quad \text{and} \quad |a_t(s,\lambda)|\leq C(\|g\|_{C^1},\phi),
		\end{equation}
		such that
		\begin{equation}\label{High}
			g(\lambda^{-1}\sqrt{x})e^{it\phi(x)}=\rho_t(x,\lambda)+t\lambda^{2m_1}\eta(\lambda^{-2}x)\int e^{ixt\lambda^{2m_1-2}s}a_t(s,\lambda)ds,
		\end{equation}
		for all $x,t>0$ and $\lambda\geq1$, where $\eta\in C^\infty(\mathbb{R})$ is supported in $[1/5,5]$ and $\eta\equiv1$ on $[1/4,4]$.
		
		In addition, if $\phi$ satisfies (H3), then
		\begin{equation}\label{better-High}
			g(\lambda^{-1}\sqrt{x})e^{it\phi(x)}=\rho_t(x,\lambda)+t^{\frac{1}{2}}\lambda^{2m_1-\alpha_1}\eta(\lambda^{-2}x)\int e^{ixt\lambda^{2m_1-2}s}a_t(s,\lambda)ds,
		\end{equation}
		for all $x,t>0$ and $\lambda\geq1$.
	\end{theorem}
	\begin{proof}
		Let $k\in\mathbb{N}$ and $t>0$. For $\lambda\geq1$, we denote by $\Psi_\lambda(\xi)$ the Fourier transform of $g(\lambda^{-1}\sqrt{x})e^{it\phi(x)}$, i.e.,
		\begin{align}\label{Fourier}
			\Psi_\lambda(\xi)&=\int g(\lambda^{-1}\sqrt{x})e^{it\phi(x)}e^{-ix\xi}dx\nonumber\\
			&=\lambda^2\int g(\sqrt{u})e^{i\left[t\phi(\lambda^2 u)-\lambda^2u\xi\right]}du.
		\end{align}
		Let $\tau\in\C^\infty(\mathbb{R})$ supported in $[2c_0^{-1},2c_0]$ with $\tau\equiv1$ in $[4c_0^{-1},c_0]$ where $c_0>2$ will be determined later. By the Fourier inversion formula, we have
		\begin{align*}
			g(\lambda^{-1}\sqrt{x})e^{it\phi(x)}&=\eta(\lambda^{-2}x)\int\left(1-\tau\left(\frac{\xi}{t\lambda^{2m_1-2}}\right)\right)\Psi_\lambda(\xi)e^{i\xi x}d\xi\\
			&\quad +\eta(\lambda^{-2}x)\int\tau\left(\frac{\xi}{t\lambda^{2m_1-2}}\right)\Psi_\lambda(\xi)e^{i\xi x}d\xi\\
			&=:\rho_t(x,\lambda)+A_{t,\lambda}(x)
		\end{align*}
		where $\eta\in C^\infty(\mathbb{R})$ is supported in $[1/5,5]$ and $\eta\equiv1$ on $[1/4,4]$.
		
		Observe that 
		\begin{equation*}
			\frac{\partial}{\partial u} \left[t\phi(\lambda^2 u)-\lambda^2u\xi\right]=\lambda^2t\phi'(\lambda^2u)-\lambda^2\xi.
		\end{equation*}
		We note that the integrand in the expression for $\rho_t(x,\lambda)$ is supported where either $\xi<4c_0^{-1}t\lambda^{2m_1-2}$ or $\xi>c_0t\lambda^{2m_1-2}$. In this situation, by (H1), we can choose $c_0$ large enough so that 
		\begin{equation*}
			\partial_u \left[t\phi(\lambda^2 u)-\lambda^2u\xi\right]\gtrsim\lambda^2|\xi|+t\lambda^{2m_1}.
		\end{equation*}
		Thus, using integration by parts in \eqref{Fourier}, we have for these $\xi$ that
		\begin{equation*}
			\left|\Psi_\lambda(\xi)\right|\leq C(k,\|g\|_{C^k},\phi) \lambda^2\left(\lambda^2|\xi|+t\lambda^{2m_1}\right)^{-k},\,\forall k\geq0.
		\end{equation*}
		which implies
		\begin{equation*}
			|\rho_t(x,\lambda)|\leq C(k,\|g\|_{C^k},\phi)\left(t\lambda^{2m_1}\right)^{-k},\,\forall k\geq0,
		\end{equation*}
		which proves \eqref{rho}.
		
		Next we estimate the term $A_{t,\lambda}(x)$. By a change of variables, we have
		\begin{align*}
			A_{t,\lambda}(x)&=\eta(\lambda^{-2}x)\int\tau\left(\frac{\xi}{t\lambda^{2m_1-2}}\right)\Psi_\lambda(\xi)e^{i\xi x}d\xi\\
			&=t\lambda^{2m_1-2}\eta(\lambda^{-2}x)\int\tau(s)\Psi_\lambda\left(t\lambda^{2m_1-2}s\right)e^{ixt\lambda^{2m_1-2}s}ds\\
			&=t\lambda^{2m_1}\eta(\lambda^{-2}x)\int a_t(s,\lambda)e^{ixt\lambda^{2m_1-2}s}ds\nonumber,
		\end{align*}
		where
		\begin{equation}\label{better-a}
			a_t(s,\lambda)=\lambda^{-2}\tau(s)\Psi_\lambda\left(t\lambda^{2m_1-2}s\right).
		\end{equation}
		It is clear that $\supp a_t(\cdot,\lambda)\in [2c_0^{-1},2c_0]$ and $|a_t(s,\lambda)|\leq C(\|g\|_{C})$, which proves \eqref{a} and also \eqref{High}.
		
		In addtion, if $\phi$ satisfies (H3), we rewrite $a_t(s,\lambda)$ in \eqref{better-a} as
		\begin{align*}
			a_t(s,\lambda)&=\lambda^{-2}\tau(s)\Psi_\lambda\left(t\lambda^{2m_1-2}s\right)\\
			&=\tau(s)\int g(\sqrt{u})e^{i\left[t\phi(\lambda^2 u)-t\lambda^{2m_1}us\right]}du.
		\end{align*}
		Noting that $g$ is supported in $[1/2,2]$, we have $u\in[1/4,4]$. For these $u$, it follows from (H3) that
		\begin{equation*}
			\left|\frac{\partial^2}{\partial u^2} \left[t\phi(\lambda^2 u)-t\lambda^{2m_1}us\right]\right|=\left|t\lambda^4\phi''(\lambda^2 u)\right|\sim t\lambda^{2\alpha_1}.
		\end{equation*}
		Therefore, by van der Corput's Lemma \ref{van-der-Corput}, we have
		\begin{equation*}
			|a_t(s,\lambda)|\leq C(\|g\|_{C^1},\phi)\left(t\lambda^{2\alpha_1}\right)^{-\frac{1}{2}}=C(\|g\|_{C^1},\phi)t^{-\frac{1}{2}}\lambda^{-\alpha_1}.
		\end{equation*}
		Replacing $a_t(s,\lambda)$ by $t^{\frac{1}{2}}\lambda^{\alpha_1}a_t(s,\lambda)$, we deduce \eqref{better-High}.
	\end{proof}
	
	In a similar way we obtained the following estimates.
	\begin{theorem} \label{Low-frequency} Assume $\phi$ satisfies (H2) and $g\in C^\infty(\mathbb{R})$ is supported in $[1/2,2]$. Then there exist $c_0>1$, and functions $\rho_t(x,\lambda)$ and $a_t(s,\lambda)$ defined on $\mathbb{R}^2$ for each $t$ satisfying
		\begin{equation*}  
			{\rm{supp}}\ \rho_t(\cdot,\lambda)\subset[\lambda^2/5,5\lambda^2]\quad \text{and} \quad |\rho_t(x,\lambda)|\leq C(k,\|g\|_{C^k},\phi)(t\lambda^{2m_2})^{-k},\, k\geq0,
		\end{equation*}
		and
		\begin{equation*}
			{\rm{supp}}\ a_t(\cdot,\lambda)\subset[2c_0^{-1},2c_0]\quad \text{and} \quad |a_t(s,\lambda)|\leq C(\|g\|_{C^1},\phi),
		\end{equation*}
		such that
		If $\phi$ satisfies (H2), then
		\begin{equation*}
			g(\lambda^{-1}\sqrt{x})e^{it\phi(x)}=\rho_t(x,\lambda)+t\lambda^{2m_2}\eta(\lambda^{-2}x)\int e^{ixt\lambda^{2m_2-2}s}a_t(s,\lambda)ds,
		\end{equation*}
		for all $x,t>0$ and $0<\lambda<1$, where $\eta\in C^\infty(\mathbb{R})$ is supported in $[1/5,5]$ and $\eta\equiv1$ on $[1/4,4]$.
		
		In addtion, if $\phi$ satisfies (H4), then
		\begin{equation*}
			g(\lambda^{-1}\sqrt{x})e^{it\phi(x)}=\rho_t(x,\lambda)+t^{\frac{1}{2}}\lambda^{2m_2-\alpha_2}\eta(\lambda^{-2}x)\int e^{ixt\lambda^{2m_2-2}s}a_t(s,\lambda)ds,
		\end{equation*}
		for all $x,t>0$ and $0<\lambda<1$.
	\end{theorem}
	\begin{proof}
		The proof of this theorem is similar to that of Theorem \ref{High-frequecy} and we omit details.
	\end{proof}
	
	\subsection{Decay estimates for $e^{it\phi(L)}$}
	In this section, we apply the subordination formulas in the previous section to obtain a high and low frequency decay estimate for the semigroup $e^{it\phi(L)}$.
	
	From Theorem \ref{High-frequecy}, we have the following high frequency decay estimate.
	\begin{theorem}\label{high-frequency-result}
		Assume $L$ satisfies (A1) and (A2),  and $\psi\in C^\infty(\mathbb{R})$ is supported in $[1/2,2]$.
		
		If $\phi$ satisfies (H1), then
		\begin{equation*} \left|\psi(\lambda^{-1}\sqrt{L})e^{it\phi(L)}f\right|\lesssim{|t|}^{-\frac{n-2}{2}}\lambda^{(1-m_1)n+2m_1}\left\|f\right\|_{L^1},\,\lambda\geq1, |t|<T_0,n\geq2;
		\end{equation*}
		In addition, if $\phi$ satisfies (H3), then
		\begin{equation*} \left|\psi(\lambda^{-1}\sqrt{L})e^{it\phi(L)}f\right|\lesssim{|t|}^{-\frac{n-1}{2}}\lambda^{(1-m_1)n+2m_1-\alpha_1}\left\|f\right\|_{L^1},\,\lambda\geq1, |t|<T_0,n\geq1.
		\end{equation*}
	\end{theorem}
	\begin{proof}
		Let $\lambda\geq 1$ and $t>0$. If $\phi$ satisfies (H1), by \eqref{High} in Theorem \ref{High-frequecy} and spectral theory, there exist functions $\rho$, $a$ and $\eta$ as in Theorem \ref{High-frequecy} such that
		\begin{align*}
			\psi(\lambda^{-1}\sqrt{L})e^{it\phi(L)}&=\rho_t(L,\lambda)+t\lambda^{2m_1}\eta(\lambda^{-2}L)\int e^{it\lambda^{2m_1-2}sL}a_t(s,\lambda)ds\\
			&=\rho_t(L,\lambda)+A_{t,\lambda}(L).
		\end{align*}
		First we estimate the term related to $A_{t,\lambda}(L)$. It follows from (A1), (A2), \eqref{L^p-L^q} and \eqref{a} that
		\begin{align*}
			\|A_{t,\lambda}(L)\|_{L^1\to L^\infty}&\lesssim t\lambda^{2m_1}\left[t\lambda^{2m_1-2}\right]^{-\frac{n}{2}}\int_{2c_0^{-1}}^{2c_0}s^{-\frac{n}{2}}|a_t(s,\lambda)|ds\\
			&\lesssim{t}^{-\frac{n-2}{2}}\lambda^{(1-m_1)n+2m_1}.
		\end{align*}
		Next we turn to the term $\rho_t(L,\lambda)$. Let $\varphi\in C^\infty(\mathbb{R})$ supported in $[1/6,6]$ and $\varphi\equiv1$ in $[1/5,5]$. Since $\rho_t(\cdot,\lambda)$ is supported in $[\lambda^2/5,5\lambda^2]$, we have
		\begin{equation*}
			\rho_t(L,\lambda)=\varphi(\lambda^{-1}\sqrt{L})\rho_t(L,\lambda)\varphi(\lambda^{-1}\sqrt{L}).
		\end{equation*}
		Hence, we get
		\begin{equation*}
			\|\rho_t(L,\lambda)\|_{L^1\to L^\infty}\leq \|\varphi(\lambda^{-1}\sqrt{L})\|_{L^1\to L^2}\|\rho_t(L,\lambda)\|_{L^2\to L^2}\|\varphi(\lambda^{-1}\sqrt{L})\|_{L^2\to L^\infty}.
		\end{equation*}
		It yields from \eqref{L^p-L^q} that
		\begin{equation}\label{L1-Linfty}
			\|\varphi(\lambda^{-1}\sqrt{L})\|_{L^1\to L^2}\lesssim \lambda^\frac{n}{2},\,\text{ and }\|\varphi(\lambda^{-1}\sqrt{L})\|_{L^2\to L^\infty}\lesssim \lambda^\frac{n}{2}.
		\end{equation}
		Using \eqref{rho}, we have
		\begin{equation}\label{L2}
			\|\rho_t(L,\lambda)\|_{L^2\to L^2}\leq \|\rho_t(\cdot,\lambda)\|_{L^\infty}\lesssim_\alpha(\lambda^{2m_1}t)^{-\alpha}, \, \forall \alpha\geq0.
		\end{equation}
		Therefore, when $n\geq2$, combined \eqref{L1-Linfty} and \eqref{L2},  taking $\alpha=\frac{n-2}{2}$, we obtain
		\begin{equation*}
			\|\rho_t(L,\lambda)\|_{L^1\to L^\infty}\lesssim \lambda^n(\lambda^{2m_1}t)^{-\frac{n-2}{2}}\lesssim t^{-\frac{n-2}{2}}\lambda^{(1-m_1)n+2m_1}.
		\end{equation*}
		%\begin{align*}\|\rho_t(L,\lambda)\|_{L^1\to L^\infty}&\lesssim \lambda^n(1+\lambda^{2m_1}t)^{-\frac{n-2}{2}}\\&\lesssim\begin{cases} t^{-\frac{n-2}{2}}\lambda^{(1-m_1)n+2m_1},&\,other\\\lambda^n,&\,1\leq n<2 \text{ and }\lambda^{2m_1}t<1.\end{cases}\end{align*}
		Summing up, we have proved that
		\begin{equation*}
			\|\psi(\lambda^{-1}\sqrt{L})e^{it\phi(L)}\|_{L^1\to L^\infty}\lesssim t^{-\frac{n-2}{2}}\lambda^{(1-m_1)n+2m_1}.
		\end{equation*}
		
		In addition, if $\phi$ satisfies (H3), by \eqref{better-High} in Theorem \ref{High-frequecy}, we have
		\begin{align*}
			\psi(\lambda^{-1}\sqrt{L})e^{it\phi(L)}&=\rho_t(L,\lambda)+t^{\frac{1}{2}}\lambda^{2m_1-\alpha_1}\eta(\lambda^{-2}L)\int e^{it\lambda^{2m_1-2}sL}a_t(s,\lambda)ds\\
			&=\rho_t(L,\lambda)+\tilde{A}_{t,\lambda}(L).
		\end{align*}
		Argued similarly as (H1), it follows from (A1), (A2), \eqref{L^p-L^q} and \eqref{a} that
		\begin{align*}
			\|A_{t,\lambda}(L)\|_{L^1\to L^\infty}&\lesssim t^{\frac{1}{2}}\lambda^{2m_1-\alpha_1}\left[t\lambda^{2m_1-2}\right]^{-\frac{n}{2}}\int_{2c_0^{-1}}^{2c_0}s^{-\frac{n}{2}}|a_t(s,\lambda)|ds\\
			&\lesssim{t}^{-\frac{n-1}{2}}\lambda^{(1-m_1)n+2m_1-\alpha_1}.
		\end{align*}
		Combined \eqref{L1-Linfty} and \eqref{L2},  taking $\alpha=\frac{n-1}{2}$ and noting the fact $\alpha_1\leq m_1$, we obtain
		\begin{align*}
			\|\rho_t(L,\lambda)\|_{L^1\to L^\infty}\lesssim \lambda^n(\lambda^{2m_1}t)^{-\frac{n-1}{2}}&=t^{-\frac{n-1}{2}}\lambda^{(1-m_1)n+m_1}\\
			&\leq {t}^{-\frac{n-1}{2}}\lambda^{(1-m_1)n+2m_1-\alpha_1},\,\forall \lambda\geq 1.
		\end{align*}
		Summing up, we have proved that
		\begin{equation*}
			\|\psi(\lambda^{-1}\sqrt{L})e^{it\phi(L)}\|_{L^1\to L^\infty}\lesssim {t}^{-\frac{n-1}{2}}\lambda^{(1-m_1)n+2m_1-\alpha_1},\,\forall \lambda\geq 1.
		\end{equation*}
	\end{proof}
	
	By Theorem \ref{Low-frequency}, we also obtain the following low frequency decay estimate.
	\begin{theorem}\label{low-frequency-result}
		Assume $L$ satisfies (A1) and (A2), and $\psi\in C^\infty(\mathbb{R})$ is supported in $[1/2,2]$.
		
		If $\phi$ satisfies (H2), then
		\begin{equation*} \left|\psi(\lambda^{-1}\sqrt{L})e^{it\phi(L)}f\right|\lesssim{|t|}^{-\frac{n-2}{2}}\lambda^{(1-m_2)n+2m_2}\left\|f\right\|_{L^1},\,0<\lambda<1, |t|<T_0, n\geq2;
		\end{equation*}
		In addition, if $\phi$ satisfies (H4), then
		\begin{equation*} \left|\psi(\lambda^{-1}\sqrt{L})e^{it\phi(L)}f\right|\lesssim{|t|}^{-\frac{n-1}{2}}\lambda^{(1-m_2)n+2m_2-\alpha_2}\left\|f\right\|_{L^1},\,0<\lambda<1, |t|<T_0, n\geq1.
		\end{equation*}
	\end{theorem}
	\begin{proof}
		The proof of this theorem is similar to that of Theorem \ref{high-frequency-result} and we omit details.
	\end{proof}
	
	\begin{proof}[Proof of Theorem \ref{main-result}]
		Combining Theorem \ref{high-frequency-result} and Theorem \ref{low-frequency-result}, we conclude Theorem \ref{main-result}.  
	\end{proof}
	
	\section{Applications}
	In this section, we apply the decay estimates in Theorem \ref{main-result} to study Strichartz estimates for some concrete equations related to Hermite operators, twisted Laplacians and Laguerre operators, whose spectrum are strictly positive and discrete. By Remark \ref{Besov-property}, we only need the decay estimates for the high frequency $\lambda\geq1$ and we see $\dot{B}_{p,q}^{s,L}$ and $B_{p,q}^{s,L}$ as one without distinction in this section.
	
	\subsection{Several concrete equations}
	We just list several concrete equations in this subsection. The fractional Schr\"odinger equation is a homogeneous example and other equations, such as the Klein-Gorden equation, the beam equation and the fourth-order Schr\"odinger equations, are non-homogeneous cases. We can not only deal with the homogeneous situations but also the non-homogeneous cases. It is worth noticing that the list of applications is not
	exhaustive since we just intend to show the generality of our approach. 
	
	We shall obtain the $B_{p',2}^{-s,L}\to B_{p,2}^{s,L}$ decay estimates for $U_t=e^{it\phi(L)}$ as an application of Theorem \ref{main-result}.  
	We assume that for $2\leq p\leq \infty$, $s=s(p)\in\mathbb{R}$ and $0<\theta=\theta(p)\leq1$,
	\begin{equation}\label{General-Assumption}
		\|U_tf\|_{B_{p,2}^{s,L}}\lesssim |t|^{-\theta} \|f\|_{B_{p',2}^{-s,L}}, \,|t|<T_0.
	\end{equation}
	%Denote $$\mathcal{A}g=\int_0^tU_{t-\tau}g(\tau,\cdot)d\tau.$$
	Above all, we introduce the following proposition to establish the general Strichartz esitimates for $U_t=e^{it\phi(L)}$. Our method is using duality argument. Since this argument is quite standard, we will omit the proof and refer the reader to \cite{KT} for details.
	\begin{proposition}\label{General-Strichartz}
		Assume $U_t$ satisfies \eqref{General-Assumption} for $p_1,p_2\geq 2$. For $0<T\leq T_0$, we have
		\begin{align*}
			\|U_tf\|_{L^\frac{2}{\theta(p_1)}\left((-T,T),B_{p_1,2}^{s(p_1),L}\right)}&\lesssim \|f\|_{L^2},\\
			\|\int_0^tU_{t-\tau}g(\tau,\cdot)d\tau\|_{L^\frac{2}{\theta(p_1)}\left((-T,T),B_{p_1,2}^{s(p_1),L}\right)}&\lesssim \|g\|_{L^{\left(\frac{2}{\theta(p_2)}\right)'}\left((-T,T),B_{p'_2,2}^{-s(p_2),L}\right)}.
		\end{align*}
	\end{proposition}
	
	\subsubsection{The fractional Schr\"odinger equation}
	We consider the fractional Schr\"{o}dinger equation ($0<\nu<1$) related to $L$
	\begin{equation}\label{FSchrEqu}
		\begin{cases}
			i\partial_tu+L^\nu u=g,\\
			u(0)=u_0.
		\end{cases}
	\end{equation}
	By Duhamel's principle, the solution is formally given by
	\begin{equation*}\label{solution1}
		u(t)=S_tu_0-i\int_0^tS_{t-\tau}g(s,\cdot)\,ds,
	\end{equation*}
	where $S_t=e^{itL^\nu}$ and it corresponds to the case when $\phi(r)=r^\nu$. By a simple calculation, 
	\begin{equation*}
		\phi'(r)=\nu r^{\nu-1}, \quad
		\phi''(r)=\nu(\nu-1)r^{\nu-2}.
	\end{equation*}
	We see that $\phi$ satisfies (H1)-(H4) with $m_1=\alpha_1=m_2=\alpha_2=\nu$.
	
	By Theorem \ref{main-result} and Plancherel formula, for $\lambda\geq1, |t|<T_0$, we have
	\begin{align*} \left\|\psi(\lambda^{-1}\sqrt{L})S_tf\right\|_{L^\infty}&\lesssim{|t|}^{-\frac{n-1}{2}}\lambda^{(1-\nu)n+\nu}\left\|f\right\|_{L^1},\\
		\left\|\psi(\lambda^{-1}\sqrt{L})S_tf\right\|_{L^2}&\lesssim\left\|f\right\|_{L^2}. 
	\end{align*}
	Using Riesz-Thorin interpolation, for $2\leq p\leq \infty$, it follows
	\begin{equation*}
		\left\|\psi(\lambda^{-1}\sqrt{L})S_tf\right\|_{L^p}\lesssim{|t|}^{-(n-1)\left(\frac{1}{2}-\frac{1}{p}\right)}\lambda^{2\left((1-\nu)n+\nu\right)\left(\frac{1}{2}-\frac{1}{p}\right)}\left\|f\right\|_{L^{p'}},\,\lambda\geq1, |t|<T_0.
	\end{equation*}
	Taking $\lambda=2^j,\forall j>0$ and because of $\psi_j(\sqrt{L})=\sum\limits_{l=-1}^1\psi_j(\sqrt{L})\psi_{j+l}(\sqrt{L})$, we obtain
	\begin{equation*}
		2^{-j\left((1-\nu)n+\nu\right)\left(\frac{1}{2}-\frac{1}{p}\right)}\left\|\psi_j(\sqrt{L})S_tf\right\|_{L^p}\lesssim{|t|}^{-(n-1)\left(\frac{1}{2}-\frac{1}{p}\right)}2^{j\left((1-\nu)n+\nu\right)\left(\frac{1}{2}-\frac{1}{p}\right)}\left\|\psi_j(\sqrt{L})f\right\|_{L^{p'}},
	\end{equation*}
	which indicates for $2\leq p\leq\infty$ that
	\begin{equation*}
		\|S_tf\|_{B_{p,2}^{s,L}}\lesssim {|t|}^{-(n-1)\left(\frac{1}{2}-\frac{1}{p}\right)}\|f\|_{B_{p',2}^{-s,L}},\, \text{with } s=-\left((1-\nu)n+\nu\right)\left(\frac{1}{2}-\frac{1}{p}\right).
	\end{equation*}
	
	Using Proposition \ref{General-Strichartz}, we obtain the Strichartz estimates for the fractional Schr\"odinger operator $S_t=e^{itL^\nu}$.
	\begin{theorem} \label{fractional} Assume $0<\nu<1$ and $n>1$. For $i=1,2$, let $p_i,q_i\geq 2$ and $s_i\in\mathbb{R}$ such that 
		\begin{equation*}
			\frac{2}{q_i}+\frac{n-1}{p_i}=\frac{n-1}{2}\quad \text{and}\quad s_i=-\left((1-\nu)n+\nu\right)\left(\frac{1}{2}-\frac{1}{p_i}\right),
		\end{equation*} 
		except $(q_i,p_i,n)=(2,\infty,3)$. Then the fractional Schr\"odinger operator $S_t=e^{itL^\nu}$ satisfies
		\begin{align*}
			\|S_tf\|_{L^{q_1}\left((-T,T),B_{p_1,2}^{s_1,L}\right)}&\lesssim \|f\|_{L^2},\\
			\|\int_0^tS_{t-\tau}g(\tau,\cdot)d\tau\|_{L^{q_1}\left((-T,T),B_{p_1,2}^{s_1,L}\right)}&\lesssim \|g\|_{L^{q'_2}\left((-T,T),B_{p'_2,2}^{-s_2,L}\right)}.
		\end{align*}
		for any $0<T\leq T_0$.
	\end{theorem}
	
	\subsubsection{The Klein-Gordon equation}
	Moreover, we consider the Klein-Gordon equation
	\begin{equation}\label{K-GEqu}
		\begin{cases}
			\partial_t^2u+L u+u=g,\\
			u(0)=u_0 ,  \\
			\partial_tu(0)=u_1.
		\end{cases}
	\end{equation}
	By Duhamel's principle, the solution is formally given by
	\begin{equation*}\label{solution3}
		u(t)=\frac{dA_t}{dt}u_0+A_tu_1-\int_0^tA_{t-\tau}g(\tau,\cdot)\,d\tau,
	\end{equation*}
	where 
	\begin{equation*}
		A_t=\frac{sin(t\sqrt{I+L})}{\sqrt{I+L}},
		\quad  \frac{dA_t}{dt}=cos(t\sqrt{I+L}).
	\end{equation*}
	So we naturally introduce the operator $K_t=e^{it\sqrt{I+L}}$, which corresponds to the case when $\phi(r)=\sqrt{1+r}$. 
	By a simple calculation,
	\begin{equation*}
		\phi'(r)=(1+r)^{-\frac{1}{2}}, \quad
		\phi''(r)=-\frac{1}{2}(1+r)^{-\frac{3}{2}}.
	\end{equation*}
	We know $\phi$ satisfies (H1)-(H4) with $m_1=\alpha_1=\frac{1}{2}, m_2=1, \alpha_2=2$. 
	
	Argued similarly as the fractional Schr\"odinger equation, for $p\geq2$, we have
	\begin{equation*}
		\|K_tf\|_{B_{p,2}^{s,L}}\lesssim {|t|}^{-(n-1)\left(\frac{1}{2}-\frac{1}{p}\right)}\|f\|_{B_{p',2}^{-s,L}},\, \text{with } s=-\frac{n+1}{2}\left(\frac{1}{2}-\frac{1}{p}\right).
	\end{equation*}
	
	Using Proposition \ref{General-Strichartz}, we obtain the Strichartz estimates for the Klein-Gordon operator $K_t=e^{it\sqrt{I+L}}$.
	\begin{theorem} \label{K-G}
		Assume $n>1$. For $i=1,2$, let $p_i,q_i\geq 2$ and $s_i\in\mathbb{R}$ such that 
		\begin{equation*}
			\frac{2}{q_i}+\frac{n-1}{p_i}=\frac{n-1}{2}\quad \text{and}\quad s_i=-\frac{n+1}{2}\left(\frac{1}{2}-\frac{1}{p_i}\right),
		\end{equation*} 
		except $(q_i,p_i,n)=(2,\infty,3)$. Then the Klein-Gordon operator $K_t=e^{it\sqrt{I+L}}$ satisfies
		\begin{align*}
			\|K_tf\|_{L^{q_1}\left((-T,T),B_{p_1,2}^{s_1,L}\right)}&\lesssim \|f\|_{L^2},\\
			\|\int_0^tK_{t-\tau}g(\tau,\cdot)d\tau\|_{L^{q_1}\left((-T,T),B_{p_1,2}^{s_1,L}\right)}&\lesssim \|g\|_{L^{q'_2}\left((-T,T),B_{p'_2,2}^{-s_2,L}\right)}.
		\end{align*}
		for any $0<T\leq T_0$.
	\end{theorem}
	\subsubsection{The beam equation}
	Besides, we consider the beam equation
	\begin{equation}\label{BeamEqu}
		\begin{cases}
			\partial_t^2u+L^2 u+u=g,\\
			u(0)=u_0 ,  \\
			\partial_tu(0)=u_1.
		\end{cases}
	\end{equation}
	By Duhamel's principle, the solution is formally given by
	\begin{equation*}\label{solution3}
		u(t)=\frac{d\mathcal{B}_t}{dt}u_0+\mathcal{B}_tu_1-\int_0^t\mathcal{B}_{t-\tau}g(\tau,\cdot)\,d\tau,
	\end{equation*}
	where 
	\begin{equation*}
		\mathcal{B}_t=\frac{sin(t\sqrt{I+L^2})}{\sqrt{I+L^2}},
		\quad  \frac{d\mathcal{B}_t}{dt}=cos(t\sqrt{I+L^2}).
	\end{equation*}
	So we naturally introduce the operator $B_t=e^{it\sqrt{I+L^2}}$, which corresponds to the case when $\phi(r)=\sqrt{1+r^2}$. 
	By a simple calculation,
	\begin{equation*}
		\phi'(r)=r(1+r^2)^{-\frac{1}{2}}, \quad
		\phi''(r)=(1+r^2)^{-\frac{3}{2}}.
	\end{equation*}
	We know $\phi$ satisfies (H1)–(H4) with $m_1=1,\alpha_1=-1, m_2=\alpha_2=2$. 
	
	By Theorem \ref{main-result} and Plancherel formula, for $\lambda\geq1, |t|<T_0$, we have
	\begin{align*} \left\|\psi(\lambda^{-1}\sqrt{L}){B}_tf\right\|_{L^\infty}&\lesssim{|t|}^{-\frac{n-1}{2}}\lambda^{3}\left\|f\right\|_{L^1},\\
		\left\|\psi(\lambda^{-1}\sqrt{L}){B}_tf\right\|_{L^2}&\lesssim\left\|f\right\|_{L^2}. 
	\end{align*}
	Taking $\lambda=2^j,\forall j>0$, using Riesz-Thorin interpolation and noting $\psi_j(\sqrt{L})=\sum\limits_{l=-1}^1\psi_j(\sqrt{L})\psi_{j+l}(\sqrt{L})$, we obtain for $p\geq2$ that
	\begin{equation*}
		\|{B}_tf\|_{B_{p,2}^{s,L}}\lesssim {|t|}^{-(n-1)\left(\frac{1}{2}-\frac{1}{p}\right)}\|f\|_{B_{p',2}^{-s,L}},\, \text{with } s=-3\left(\frac{1}{2}-\frac{1}{p}\right).
	\end{equation*}
	
	Using Proposition \ref{General-Strichartz}, we obtain the Strichartz estimates for the beam operator $B_t=e^{it\sqrt{I+L^2}}$.
	\begin{theorem}\label{beam} Assume $n>1$. For $i=1,2$, let $p_i,q_i\geq 2$ and $s_i\in\mathbb{R}$ such that 
		\begin{equation*}
			\frac{2}{q_i}+\frac{n-1}{p_i}=\frac{n-1}{2}\quad \text{and}\quad s_i=-3\left(\frac{1}{2}-\frac{1}{p_i}\right),
		\end{equation*} 
		except $(q_i,p_i,n)=(2,\infty,3)$. Then the beam operator $B_t=e^{it\sqrt{I+L^2}}$ satisfies
		\begin{align*}
			\|B_tf\|_{L^{q_1}\left((-T,T),B_{p_1,2}^{s_1,L}\right)}&\lesssim \|f\|_{L^2},\\
			\|\int_0^tB_{t-\tau}g(\tau,\cdot)d\tau\|_{L^{q_1}\left((-T,T),B_{p_1,2}^{s_1,L}\right)}&\lesssim \|g\|_{L^{q'_2}\left((-T,T),B_{p'_2,2}^{-s_2,L}\right)}.
		\end{align*}
		for any $0<T\leq T_0$.
	\end{theorem}
	
	\subsubsection{The fourth-order Schr\"{o}dinger equation}
	Finally, we consider the fourth-order Schr\"{o}dinger equation
	\begin{equation}\label{SSchrEqu}
		\begin{cases}
			i\partial_tu+L^2u+Lu=g,\\
			u(0)=u_0.
		\end{cases}
	\end{equation}
	By Duhamel's principle, the solution is formally given by
	\begin{equation*}\label{solution3}
		u(t)=\mathcal{U}_tu_0-i\int_0^t\mathcal{U}_{t-\tau}g(\tau,\cdot)\,d\tau,
	\end{equation*}
	where $\mathcal{U}_t=e^{it(L^2+L)}$ and it corresponds to the case when $\phi(r)=r^2+r$. By a simple calculation,
	\begin{equation*}
		\phi'(r)=2r+1, \quad
		\phi''(r)=2.
	\end{equation*}
	We know $\phi$ satisfies (H1)-(H4) with $m_1=\alpha_1=\alpha_2=2$, $m_2=1$. 
	
	Argued similarly as the fractional Schr\"odinger equation, for $p\geq2$, we have
	\begin{equation*}
		\|\mathcal{U}_tf\|_{B_{p,2}^{s,L}}\lesssim {|t|}^{-(n-1)\left(\frac{1}{2}-\frac{1}{p}\right)}\|f\|_{B_{p',2}^{-s,L}},\, \text{with } s=-(2-n)\left(\frac{1}{2}-\frac{1}{p}\right).
	\end{equation*}
	
	Using Proposition \ref{General-Strichartz}, we obtain the Strichartz estimates for the fourth-order Schr\"odinger operator $\mathcal{U}_t=e^{it(L^2+L)}$.
	\begin{theorem} \label{fourth}
		Assume $n>1$. For $i=1,2$, let $p_i,q_i\geq 2$ and $s_i\in\mathbb{R}$ such that 
		\begin{equation*}
			\frac{2}{q_i}+\frac{n-1}{p_i}=\frac{n-1}{2}\quad \text{and}\quad s_i=-(2-n)\left(\frac{1}{2}-\frac{1}{p_i}\right),
		\end{equation*} 
		except $(q_i,p_i,n)=(2,\infty,3)$. Then the fourth-order Schr\"odinger operator $\mathcal{U}_t=e^{it(L^2+L)}$ satisfies
		\begin{align*}
			\|\mathcal{U}_tf\|_{L^{q_1}\left((-T,T),B_{p_1,2}^{s_1,L}\right)}&\lesssim \|f\|_{L^2},\\
			\|\int_0^t\mathcal{U}_{t-\tau}g(\tau,\cdot)d\tau\|_{L^{q_1}\left((-T,T),B_{p_1,2}^{s_1,L}\right)}&\lesssim \|g\|_{L^{q'_2}\left((-T,T),B_{p'_2,2}^{-s_2,L}\right)}.
		\end{align*}
		for any $0<T\leq T_0$.
	\end{theorem}
	
	\subsection{The Hermite operator}
	Let $H=-\Delta +|x|^2$ be the Hermite operator on $\mathbb{R}^n, \,n\geq1$. For $k\in\mathbb{N}$, let $H_k$ denote the Hermite polynomial on $\mathbb{R}$ defined by
	\begin{equation*}
		H_k(x)=(-1)^k\frac{d^k}{dx^k}(e^{-x^2})e^{x^2},
	\end{equation*}
	and $h_k$ denote the normalized Hermite function on $\mathbb{R}$ defined by
	\begin{equation*}
		h_k(x)=(2^k\sqrt{\pi}k!)^{-\frac{1}{2}}H_k(x)e^{-\frac{1}{2}x^2}.
	\end{equation*}
	For any multi-index $\beta\in\mathbb{N}^n$ and $x=(x_1,x_2,\cdots,x_n)\in\mathbb{R}^n$, We set $\Phi_\beta(x)=\prod\limits_{k=1}^nh_{\beta_k}(x_k)$. It is well-known that the family $\{\Phi_\beta\}_{\beta\in\mathbb{N}^n}$ forms a complete orthonormal basis for $L^2(\mathbb{R}^n)$ and they are eigenfunctions of the Hermite operator $H$ corresponding to eigenvalues $2|\beta|+n$, where $|\beta|=\sum\limits_{k=1}^n\beta_k.$ Then the Hermite operator is a strictly positive self-adjoint operator on $L^2(\mathbb{R}^n)$. For every $f\in L^2(\mathbb{R}^n)$, it has the expansion
	\begin{equation}\label{twisted-expansion}
		f=\sum_{\beta\in\mathbb{N}^n}\langle f,\Phi_{\beta}\rangle \Phi_{\beta}=\sum_{k=1}^\infty P_kf,
	\end{equation}
	where 
	\begin{equation*}
		P_kf=\sum_{|\beta|=k}\langle f,\Phi_\beta\rangle \Phi_\beta
	\end{equation*}
	denotes the orthogonal projection of $L^2(\mathbb{R}^n)$ onto the eigenspace spanned by $\{\Phi_\beta: |\beta|=k\}$.
	
	The semigroup $\{e^{-tH}: t>0\}$ associated to the Hermite operator $H$ is defined by
	\begin{equation*}
		e^{-tH}f=\sum_{k=0}^\infty e^{-t(2k+n)}P_kf,
	\end{equation*}
	for $f\in L^2(\mathbb{R}^n)$ and the above expression can be written as 
	\begin{equation*}
		e^{-tH}f=\int_{\mathbb{R}^n}p_t(x,y)f(y)dy,
	\end{equation*}
	where the kernel $p_t(x,y)$ is given by the expansion
	\begin{equation*}
		p_t(x,y)=\sum_{k=0}^\infty e^{-t(2k+n)}\sum_{|\beta|=k}\Phi_\beta(x)\Phi_\beta(y)=\sum_{\beta\in\mathbb{N}^n}e^{-t(2|\beta|+n)}\Phi_\beta(x)\Phi_\beta(y).
	\end{equation*}
	Using Mehler's formula for Hermite functions, the kernel $p_t(x,y)$ can be expressed explicitly by 
	\begin{align*}
		p_t(x,y)&=\left(2\pi \sinh 2t\right)^{-\frac{n}{2}}\exp \left(-\frac{1}{2}\left(|x|^2+y|^2\right)\coth 2t+\frac{x\cdot y}{\sinh 2t}\right)\\
		&=\left(2\pi \sinh 2t\right)^{-\frac{n}{2}}\exp \left(-\frac{1}{4}\frac{1+e^{-2t}}{1-e^{-2t}}|x-y|^2-\frac{1}{4}\frac{1-e^{-2t}}{1+e^{-2t}}|x+y|^2\right)
	\end{align*}
	for all $t>0$ and $x,y\in\mathbb{R}^n$. It is easy to see that
	\begin{equation*}
		0< p_t(x,y)\leq \left(4\pi t\right)^{-\frac{n}{2}}\exp \left(-\frac{|x-y|^2}{4t}\right),
	\end{equation*}
	which indicates that $p_t(x,y)$ enjoys the Gaussion upper bound (A2).
	
	Moreover, the kernel of the Schr\"odinger operator $e^{-itH}$ can be written as 
	\begin{equation*}
		p_{it}(x,y)=\left(2\pi i \sin 2t\right)^{-\frac{n}{2}}\exp \left(\frac{i}{2}\left(|x|^2+|y|^2\right)\cot 2t-\frac{ix\cdot y}{\sin 2t}\right)
	\end{equation*}
	and 
	\begin{equation*}
		|p_{it}(x,y)|\leq\left(2\pi |\sin 2t|\right)^{-\frac{n}{2}}\lesssim |t|^{-\frac{n}{2}},\,\forall |t|\leq\frac{\pi}{4},
	\end{equation*}
	which yields that
	\begin{equation*}
		\|e^{itH}f\|_{L^1\to L^\infty}\lesssim |t|^{-\frac{n}{2}},\,\forall |t|\leq\frac{\pi}{4}.
	\end{equation*}
	We refer to \cite{NR2005, Thangavelu} for a detailed study on Hermite functions and the kernels associated to $e^{-tH}$ and $e^{-itH}$.
	Therefore, the Hermite operator $H$ satisfies the conditions (A1) and (A2). The estimates in Theorem \ref{fractional}, \ref{K-G}, \ref{beam} and \ref{fourth} hold true for the Hermite operator when $n\geq2$.
	\subsection{The twisted Laplacian}
	The twisted Laplacian (or the special Hermite operator), introduced by Strichartz \cite{Strichartz} is defined on $\mathbb{C}^d=\mathbb{R}^{n}$ ($n=2d$) by
	\begin{equation*}
		L=\frac{1}{2}\sum\limits_{k=1}^d(Z_k\overline{Z}_k+\overline{Z}_kZ_k),\quad \text{where} \quad Z_k=\frac{\partial}{\partial {z_k}}+\frac{\overline{z}_k}{2}, \,\overline{Z}_k=-\frac{\partial}{\partial\overline{z}_k}+\frac{{z}_k}{2},\,k=1,2,\cdots,d.
	\end{equation*}
	Here $\frac{\partial}{\partial {z_k}}$ and $\frac{\partial}{\partial\overline{z}_k}$ denote the complex derivatives $\frac{\partial}{\partial {x_k}}\mp i\frac{\partial}{\partial {y_k}}$, respectively. In explicit terms it has the form 
	\begin{equation*}
		L=-\Delta_z+\frac{1}{4}|z|^2-i\sum_{k=1}^d\left(x_k\frac{\partial}{\partial {y_k}}-y_k\frac{\partial}{\partial {x_k}}\right),
	\end{equation*}
	where $\Delta_z$ is the Laplacian on $\mathbb{C}^d$. For each multi-index $\mu,\nu\in\mathbb{N}^d$, we define the special Hermite functions on $\mathbb{C}^d$ by
	\begin{equation*}
		\Phi_{\mu,\nu}(z)=(2\pi)^{-\frac{d}{2}}\int_{\mathbb{R}^d}e^{ix\cdot\xi}\Phi_\mu\left(\xi+\frac{y}{2}\right)\Phi_\nu\left(\xi-\frac{y}{2}\right)d\xi,\,z=x+iy\in\mathbb{C}^d.
	\end{equation*}
	The family $\{\Phi_{\mu,\nu}: \mu,\nu\in\mathbb{N}^d\}$ forms a complete orthonormal basis for $L^2(\mathbb{C}^d)$ and they are eigenfunctions of the twisted Laplacian corresponding to eigenvalues $2|\nu|+d$. Then the twisted Laplacian is also a positive self-adjoint operator on $L^2(\mathbb{C}^d)$. For every $f\in L^2(\mathbb{C}^d)$, it has the expansion
	\begin{equation}\label{twisted-expansion}
		f=\sum_{\mu,\nu\in\mathbb{N}^d}\langle f,\Phi_{\mu,\nu}\rangle \Phi_{\mu,\nu}=\sum_{k=1}^\infty P_kf
	\end{equation}
	where
	\begin{equation*}
		P_kf=\sum_{\mu\in\mathbb{N}^d,|\nu|=k}\langle f,\Phi_{\mu,\nu}\rangle \Phi_{\mu,\nu}
	\end{equation*}
	is the spectral projection corresponding to the eigenvalue $2k+d$.
	
	The twisted convolution of two functions $f$ and $g$ on $\mathbb{C}^d$ is defined by
	\begin{equation*}
		f\times g(z)=\int_{\mathbb{C}^d} f(w)g(z-w)e^{-\frac{i}{2}Im (z\cdot\overline{w})} dw.
	\end{equation*}
	The special Hermite functions satisfy the following orthogonality properties
	\begin{equation}\label{orthogonality}
		\Phi_{\mu,\nu}\times \Phi_{\mu',\nu'}=(2\pi)^\frac{n}{2} \delta_{\nu,\mu'}\Phi_{\mu,\nu'},
	\end{equation}
	where $\delta_{\nu,\mu'}=1$ if $\nu=\mu'$ and $0$ otherwise. Let $\varphi_k^{(d-1)}(z)=L_k^{(d-1)}\left(\frac{|z|^2}{2}\right)e^{-\frac{|z|^2}{4}}$ be the Laguerre function of order $d-1$, where $L_k^{(\alpha)}$ denotes the Laguerre polynomial of degree $k$ and order $\alpha>-1$, defined on $\mathbb{R}$ by the generating identity (see \cite{Lebedev})
	\begin{equation}\label{generating}
		\sum_{k=1}^\infty L_k^{(\alpha)}(x)\omega^k=(1-\omega)^{-\alpha-1}e^{-\frac{\omega x}{1-\omega}}, \,|\omega|<1.
	\end{equation}
	The special Hermite functions $\Phi_{\mu,\nu}$ are related to the Laguerre functions $\varphi_k^{(d-1)}$ by the following relation (see \cite{Thangavelu})
	\begin{equation}\label{relationship}
		\varphi_k^{(d-1)}=(2\pi)^\frac{d}{2}\sum_{|\nu|=k}\Phi_{\nu,\nu}.
	\end{equation}
	Taking twisted convolution on both sides of \eqref{twisted-expansion} with $\Phi_{\nu,\nu}$ and using the orthonogality property \eqref{orthogonality}, we have
	\begin{equation*}
		f\times \Phi_{\nu,\nu}=(2\pi)^\frac{d}{2}\sum_{\mu\in\mathbb{N}^d}\langle f,\Phi_{\mu,\nu}\rangle \Phi_{\mu,\nu}.
	\end{equation*}
	Summing both sides of the above equation with respect to all $|\nu|=k$ and by \eqref{relationship}, we get
	\begin{equation*}
		P_kf=(2\pi)^{-\frac{d}{2}}\sum_{|\nu|=k}f\times \Phi_{\nu,\nu}=(2\pi)^{-d}f\times\varphi_k^{(d-1)}.
	\end{equation*}
	Hence, the special Hermite expansion takes the compact form
	\begin{equation*}
		f=(2\pi)^{-d}\sum_{k=0}^\infty f\times\varphi_k^{(d-1)}.
	\end{equation*}
	
	The semigroup $\{e^{-tL}: t>0\}$ associated to the twisted Laplacian $L$ is defined by
	\begin{equation*}
		e^{-tL}f=\sum_{k=0}^\infty e^{-t(2k+d)}P_kf=(2\pi)^{-d}\sum_{k=0}^\infty e^{-t(2k+d)}f\times\varphi_k^{(d-1)},
	\end{equation*}
	for $f\in L^2(\mathbb{C}^d)$. The above expression can be expressed as a twisted convolution 
	\begin{equation*}
		e^{-tL}f(z)=f\times K_t(z)=\int_{\mathbb{C}^d} f(w)K_t(z-w)e^{-\frac{i}{2}Im (z\cdot\overline{w})} dw=\int_{\mathbb{C}^d} p_t(z,w)f(w)dw,
	\end{equation*}
	where $K_t$ and the kernel $p_t(z,w)$ are defined respectively by
	\begin{equation*}
		K_t(z)=(2\pi)^{-d}\sum_{k=0}^\infty e^{-t(2k+d)}\varphi_k^{(d-1)}(z),\quad p_t(z,w)=K_t(z-w)e^{-\frac{i}{2}Im (z\cdot\overline{w})}.
	\end{equation*}
	By the generating identity \eqref{generating}, we have
	\begin{equation*}
		K_t(z)=(4\pi \sinh t)^{-d} \exp\left(-\frac{
			|z|^2
		}{4}\coth t\right),
	\end{equation*}
	which indicates that the kernel $p_t(z,w)$ satisfies the Gaussian upper bound (A2):
	\begin{equation*}
		|p_t(z,w)|=|K_t(z-w)|\leq \left(4\pi t\right)^{-d}\exp \left(-\frac{|x-y|^2}{4t}\right).
	\end{equation*}
	
	Moreover, the kernel of the Schr\"odinger operator $e^{-itL}$ can be written as 
	\begin{equation*}
		p_{it}(z,w)=(4\pi i\sin t)^{-d} \exp\left(i\frac{
			|z-w|^2
		}{4}\cot t\right),
	\end{equation*}
	and 
	\begin{equation*}
		|p_{it}(z,w)|\lesssim|\sin t|^{-d}\lesssim |t|^{-d},\,\forall |t|\leq\frac{\pi}{2},
	\end{equation*}
	which yields that
	\begin{equation*}
		\|e^{itL}f\|_{L^1\to L^\infty}\lesssim |t|^{-\frac{n}{2}},\,\forall |t|\leq\frac{\pi}{2},\text{ where } n=2d.
	\end{equation*}
	We refer to \cite{R2008, Thangavelu} for a detailed study on the special Hermite functions and the kernels associated to $e^{-tL}$ and $e^{-itL}$.
	Therefore, the Hermite operator $L$ satisfies the conditions (A1) and (A2). The estimates in Theorem \ref{fractional}, \ref{K-G}, \ref{beam} and \ref{fourth} hold true for the twisted Laplacian when $d\geq1$.
	
	\subsection{The Laguerre operator}
	For multi-index $\alpha\in(-1,\infty)^d$, consider the Laguerre operator $L_\alpha$ defined on $\mathbb{R}_+^d=(0,\infty)^d$ by
	\begin{equation*}
		L_{\alpha}=-\Delta-\sum_{k=1}^{d}\frac{2\alpha_k+1}{x_k}\dfrac{\partial}{\partial x_k}+\dfrac{|x|^2}{4}.
	\end{equation*}
	The space $\mathbb{R}_+^d=(0,\infty)^d$ is equipped with the Euclidean distance $d$ and measure $\mu_\alpha$ given by
	$d\mu_\alpha(x)=x_1^{2\alpha_1+1}x_2^{2\alpha_2+1}\dots x_n^{2\alpha_n+1}dx_1dx_2\dots dx_n=x^{2\alpha+1}dx$.
	For any $x\in \mathbb{R}_+^d$ and $r>0$, we denote by
	\begin{equation*}
		B(x,r)=\{y\in \mathbb{R}_+^d: |x-y|<r\}
	\end{equation*}
	the ball centered in $x$ with radius $r$. It is easy to see that
	\begin{equation*}
		\mu_\alpha(B(x,r))\sim \prod_{k=1}^d r(r+x_k)^{2\alpha_k+1},
	\end{equation*}
	which follows that the measure $\mu_\alpha$ satisfies the doubling condition \eqref{doubling-condition}. Moreover, if $\alpha\in(-\frac{1}{2},\infty)^d$, we have for all $x\in \mathbb{R}_+^d$ and $r>0$
	\begin{equation*}
		\mu_\alpha(B(x,r))\gtrsim r^n, \quad\text{where}\quad n=2\left(d+\sum\limits_{k=1}^d\alpha_k\right).
	\end{equation*}
	Therefore, throughout this section we always assume that $\alpha\in(-\frac{1}{2},\infty)^d$.
	For any multi-index $\beta\in\mathbb{N}^d$, define the Laguerre functions on $\mathbb{R}_+^d=(0,\infty)^d$ by
	\begin{equation*}
		\psi_\beta^{(\alpha)}(x)=\prod_{k=1}^{d}\psi_{\beta_k}^{(\alpha_j)}(x_k),\, \forall x\in{\mathbb{R}_+^n},
	\end{equation*}
	where the Laguerre function $\psi_{\beta_k}^{(\alpha_j)}(x_k)$ on $\mathbb{R_+}$ is given by
	\begin{equation*}
		\psi_{\beta_k}^{(\alpha_j)}(x_k)={\left(\frac{2^{-\alpha_k}\Gamma(k+1)}{\Gamma(k+\alpha_k+1)}\right)}^{\frac{1}{2}}L_k^{(\alpha_k)}\left(\frac{x_k^2}{2}\right)e^{-\frac{x_k^2}{4}}.
	\end{equation*}
	The family $\{\psi_\beta^{(\alpha)}: \beta\in\mathbb{N}^d\}$ forms a complete orthonormal basis for $L^2(\mathbb{R}_+^n,d\mu_\alpha)$ and they are eigenfunctions of the Laguerre operator $L_\alpha$ corresponding to eigenvalues $2|\beta|+\sum\limits_{k=1}^{d}\alpha_j+d>\frac{d}{2}$.
	Then the Laguerre operator is a strictly positive self-adjoint operator on $L^2(\mathbb{R}_+^n,d\mu_\alpha)$. For any $f\in L^2(\mathbb{R}_+^n,d\mu_\alpha)$, it has the Laguerre expansion 
	\begin{equation*}
		f=\sum_{\beta\in\mathbb{N}^d}\langle f,\psi_\beta^{(\alpha)}\rangle_\alpha\psi_\beta^{(\alpha)}=\sum_{k=0}^\infty P_k f
	\end{equation*}
	where $\langle\cdot,\cdot\rangle_\alpha$ is an inner product inherited from $L^2(\mathbb{R}^n_+,d\mu_\alpha)$ and
	\begin{equation*}
		P_k f=\sum_{|\beta|=k}\langle f,\psi_\beta^{(\alpha)}\rangle_\alpha\psi_\beta^{(\alpha)} 
	\end{equation*}
	denotes the Laguerre projection operator corresponding to the eigenvalue $2k+\sum\limits_{k=1}^{d}\alpha_j+d$.

	The semigroup $\{e^{-tL_\alpha}: t>0\}$ associated to the Laguerre operator $L_\alpha$ is defined by
	\begin{equation*}
		e^{-tL_\alpha}f=\sum_{k=0}^\infty e^{-t(2k+\sum\limits_{k=1}^{d}\alpha_j+d)}P_kf,
	\end{equation*}
	for $f\in L^2(\mathbb{R}^n)$ and the above expression can be written as 
	\begin{equation*}
		e^{-tL}f=\int_{\mathbb{R}^n}p_t(x,y)f(y)d\mu_\alpha(y),
	\end{equation*}
	where the kernel $p_t(x,y)$ is given by the expansion
	\begin{equation*}
		p_t(x,y)=\sum_{k=0}^\infty e^{-t(2k+\sum\limits_{k=1}^{d}\alpha_j+d)}\sum_{|\beta|=k}\psi_\beta^{(\alpha)}(x)\psi_\beta^{(\alpha)}(y)=\sum_{\beta\in\mathbb{N}^n}e^{-t(2|\beta|+\sum\limits_{k=1}^{d}\alpha_j+d)}\psi_\beta^{(\alpha)}(x)\psi_\beta^{(\alpha)}(y).
	\end{equation*}
	Using Mehler's formula for Laguerre functions, the kernel $p_t(x,y)$ can be expressed explicitly by 
	\begin{align*}
		p_t(x,y)&=\prod_{k=1}^d\frac{2e^{-2t}}{1-e^{-4t}}\exp\left(-\frac{1}{2}\frac{1+e^{-4t}}{1-e^{-4t}}(x_k^2+y_k^2)\right) (x_ky_k)^{-\alpha_k} I_{\alpha_k}\left(\frac{2e^{-2t}}{1-e^{-4t}} x_ky_k\right)
	\end{align*}
	for all $t>0$,$x,y\in\mathbb{R}_+^d$ and $I_{\alpha_k}$ being the bessel funtion. $p_t(x,y)$ enjoys the Gaussion upper bound (A2) (see \cite{Lebedev, BD}),i.e, there exist $C,c>0$ such that for all $x,y\in \mathbb{R}_+^d$ and $t>0$,
	\begin{equation*}
		0<p_t(x,y)\lesssim \frac{C}{\mu_\alpha(B(x,\sqrt{t})}\exp \left(-\frac{|x-y|^2}{ct}\right).
	\end{equation*}
	
	Moreover, the kernel of the Schr\"odinger operator $e^{-itL_\alpha}$ satisfies the following estimate (see \cite{S2013})
	\begin{equation*}
		|p_{it}(x,y)|\lesssim|\sin t|^{-\left(d+\sum\limits_{k=1}^d\alpha_k\right)}\lesssim |t|^{-\left(d+\sum\limits_{k=1}^d\alpha_k\right)},\,\forall |t|\leq\frac{\pi}{2},
	\end{equation*}
	which yields that
	\begin{equation*}
		\|e^{itL_\alpha}f\|_{L^1\to L^\infty}\lesssim |t|^{-\frac{n}{2}},\,\forall |t|\leq\frac{\pi}{2}, \quad\text{where}\quad n=2\left(d+\sum\limits_{k=1}^d\alpha_k\right)>d\geq1.
	\end{equation*}
	Therefore, the Laguerre operator $L_\alpha$ satisfies the conditions (A1) and (A2). The estimates in Theorem \ref{fractional}, \ref{K-G}, \ref{beam} and \ref{fourth} hold true for the Laguerre operator when $d\geq1$.

	\section*{Acknowledgments}The second author is supported by the National Natural Science Foundation of China (Grant No. 11701452) and Guangdong Basic and Applied Basic Research Foundation (No. 2023A1515010656). The third author is supported by the National Natural Science Foundation of China (Grant No. 12171399 and 12271401).


\begin{thebibliography}{99}
		
		
		\normalsize
		\baselineskip=17pt
		\bibitem{BBG2021}\label{BBG2001} H. Bahouri, D. Barilari, I. Gallagher, \emph{Strichartz estimates and Fourier restriction theorems on the Heisenberg group}, J. Fourier Anal. Appl. \textbf{27}(2) (2021).
		
		\bibitem{BGX2000}\label{BGX2000} H. Bahouri, P. G\'erard, C. J. Xu, \emph{Espaces de Besov et estimations de Strichartz g\'{e}n\'{e}ralis\'{e}es sur le groupe de
			Heisenberg}, J. Anal. Math. \textbf{82}, 93-118 (2000).
		
		\bibitem{BKG}\label{BKG} H. Bahouri, C. F. Kammerer,  I. Gallagher, \emph{Dispersive estimates for the Schr\"{o}dinger operator on step 2 stratified Lie groups}, Anal. PDE. \textbf{9}, 545-574 (2016).
		
		\bibitem{BD}T. A. Bui, X. T. Duong, \emph{Laguerre operator and its associated weighted Besov and Triebel–Lizorkin spaces}, Trans. Am. Math. Soc. \textbf{369}(3), 2109-2150 (2017).
		
		\bibitem{BDXM2019}\label{BDXM2019} T. A. Bui, P. D\text{'}Ancona,  X. T. Duong, D. M\"{u}ller, \emph{On the flows associated to self-adjoint operators on metric measure spaces}, Math. Ann. \textbf{375}, 1393-1426 (2019).
		
		\bibitem{DPR2010}\label{DPR2010} P. D\text{'}Ancona, V. Pierfelice, F. Ricci, \emph{On the wave equation associated to the Hermite and the twisted Laplacian}, J. Fourier Anal. Appl. \textbf{16}(2), 294-310 (2010).
		
		\bibitem{FMV1}\label{FMV1}G. Furioli, C. Melzi and A. Veneruso, \emph{Strichartz inequalities for the wave equation with the full Laplacian on the Heisenberg group},  Canad. J. Math. \textbf{59}(6), 1301-1322 (2007)
		
		\bibitem{FV}\label{FV} G. Furioli and A. Veneruso, \emph{Strichartz inequalities for the Schr\"{o}dinger equation with the full Laplacian on the Heisenberg group}, Studia Math. \textbf{160}, 157-178 (2004)
		
		\bibitem{GV}J. Ginibre, G. Velo, \emph{Generalized Strichartz inequalities for the wave equation}, J. Funct. Anal. \textbf{133}, 50-68 (1995). 
		
		\bibitem{GPW2008} Z. Guo, L. Peng and B. Wang, \emph{Decay estimates for a class of wave equations}, J. Funct. Anal. \textbf{254}(6), 1642-1660 (2008). 
		
		\bibitem{Lebedev}\label{Lebedev} N. N. Lebedev, \emph{Special Functions and Their Applications}, Dover, NewYork (1972).
		\bibitem{LS2014} H. Liu and M. Song, \emph{Strichartz inequalities for the wave equation with the full Laplacian on H-type groups}, Abstr. Appl. Anal., Art. ID 219375, 10 pp (2014).
		
		\bibitem{H2005}\label{H2005} M. D. Hierro, \emph{Dispersive and Strichartz estimates on H-type groups}, Studia Math. \textbf{169}, 1-20 (2005). 
		
		\bibitem{Hulanicki}\label{Hulanicki} A. Hulanicki, \emph{A functional calculus for Rockland operators on nilpotent Lie groups}, Studia Math. \textbf{78}, 253-266 (1984).
		
		\bibitem{KT}\label{KT} M. Keel, T. Tao, \emph{Endpoints Strichartz estimates}, Amer. J. Math. \textbf{120}, 955-980 (1998).
		
		\bibitem{Muller-Seeger2015}\label{Muller-Seeger2015} D. M\"uller, A. Seeger, \emph{Sharp $L^p$ bounds for the wave equation on groups of Heisenberg type}, Anal. PDE \textbf{8}(5), 1051-1100 (2015).
		
		\bibitem{NR2005}\label{NR2005} A. K. Nandakumaran, P. K. Ratnakumar, \emph{Schr\"{o}dinger equation and the oscillatory semigroup for the Hermite operator}, J. Funct. Anal. \textbf{224}(2), 371-385 (2021).
		
		\bibitem{R2008}\label{R2008} P. K. Ratnakumar, \emph{On Schr\"{o}dinger Propagator for the Special Hermite Operator}, J. Fourier Anal. Appl. \textbf{14}(2), 286-300 (2008).
		
		
		\bibitem{S2013}V. K. Sohani, \emph{Strichartz estimates for the Schr\"{o}dinger propagator for the Laguerre operator}, Proc. Indian Acad. Sci. (Math. Sci.) \textbf{123}, 525-537 (2013).
		
		\bibitem{Song2016}M. Song, \emph{Decay estimates for fractional wave equations on H-type groups}, J. Inequal. Appl. \textbf{2016}(246), 12pp (2016).
		
		\bibitem{SY2023}N. Song, J. Yang, \emph{Decay estimates for a class of wave equations on the Heisenberg group}, Ann. Mat. Pur. Appl., https://doi.org/10.1007/s10231-023-01334-x (2023).
		
		\bibitem{SZ}N. Song, J. Zhao, \emph{Strichartz estimates on the quaternion Heisenberg group}, Bull. Sci. Math. \textbf{138}(2), 293-315 (2014).
		
		\bibitem{S1993}\label{S1993} E. M. Stein,\emph{Harmonic analysis: real-variable methods, orthogonality and oscillatory integrals}, Princeton Univ. Press, 1993. 
		
		\bibitem{Str}\label{Str} R. S. Strichartz, \emph{Restrictions of Fourier transforms to quadratic surfaces and
			decay of solutions of wave equations}, Duke Math. J. \textbf{44}, 705-714 (1977).
		
		\bibitem{Strichartz}\label{Strichartz} R. S. Strichartz, \emph{Harmonic analysis as spectral theory of Laplacians}, J. Funct. Anal. \textbf{87}, 51-148 (1989). 
		
		\bibitem{Thangavelu}\label{Thangavelu} S. Thangavelu, \emph{Lectures on Hermite and Laguerre expansions}, Princeton Univ. Press, 1993.
		
	\end{thebibliography}
\end{document}